\documentclass{amsart}
\usepackage{amssymb,latexsym,amscd,hyperref}
\usepackage{tikz-cd}
\usepackage{pb-diagram}
\theoremstyle{plain}
\newtheorem{theorem}{Theorem}
\newtheorem{corollary}[theorem]{Corollary}
\newtheorem{lemma}[theorem]{Lemma}

\newtheorem{remark}[theorem]{Remark}
\newtheorem{proposition}[theorem]{Proposition}
\newtheorem{algorithm}[theorem]{Algorithm}

\newtheorem{conjecture}{Conjecture}

\theoremstyle{definition}
\newtheorem{definition}[theorem]{Definition}

\theoremstyle{remark}

\numberwithin{theorem}{section} 
\newcommand {\Q}{{\mathbb{Q}}}

\newcommand {\Z}{{\mathbb{Z}}}
\newcommand {\N}{{\mathbb{N}}}
\newcommand {\F}{{\mathbb{F}}}

\newcommand {\OO}{{\mathcal{O}}}

\newcommand {\idf}{{\mathfrak{f}}}
\newcommand {\idp}{{\mathfrak{p}}}

\newcommand{\disc}     {{\mathfrak d}}%
\newcommand{\fm}     {{\mathfrak m}}%
\newcommand{\Disc}     {\mathop{\rm {Disc}}}%
\newcommand{\normal}     {\unlhd}%
\newcommand{\id}     {\mathop{\rm {id}}}%
\newcommand{\Gal}      {\mathop{\rm {Gal}}}
\newcommand{\Syl}      {\mathop{\rm {Syl}}}
\newcommand{\Nm}      {\mathop{\mathcal{N}}}

\newcommand{\Cl}      {\mathop{\rm {Cl}}\nolimits}

\newcommand{\ind}        {{\mathop{\rm ind}}}
\newcommand{\rk}        {{\mathop{\rm rk}}}
\newcommand{\eps}  {{\epsilon}}

\begin{document}

\title[$\ell$-torsion bounds]{$\ell$-torsion bounds for the class group of number fields with an  $\ell$-group as Galois group}

\author{J\"urgen Kl\"uners} \address{Universit\"at Paderborn, Institut f\"ur 
  Mathematik, Warburger Str. 100, 33098 Paderborn, Germany}

\email{klueners@math.uni-paderborn.de}

\author{Jiuya Wang}
\address{Duke University, Department of Mathematics, Durham, NC, 27708, US}
\email{wangjiuy@math.duke.edu}
\begin{abstract} 
	We describe the relations among the $\ell$-torsion conjecture, a conjecture of Malle giving an upper bound for the number of extensions, and the discriminant multiplicity conjecture. 
	We prove that the latter two conjectures are equivalent in some sense. Altogether, the three conjectures are equivalent for the class of solvable groups.
		We then prove the $\ell$-torsion conjecture for $\ell$-groups and the other two conjectures for nilpotent groups.
\end{abstract}

\maketitle
 
\renewcommand{\theenumi}{(\roman{enumi})}

\section{Introduction}

In this paper, we study several conjectures in arithmetic statistics. In all situations $E/F$ will be an extension of degree $n$ of number fields with absolute discriminant $\Disc(E/F)$ which is the absolute norm of the discriminant ideal $\disc(E/F)$.

\begin{conjecture}[$\ell$-torsion conjecture, \cite{BruSil,Duk98,Zha05}]\label{conj:ell-torsion}
	 Let $\ell$ be a prime number and $E/F$ be an extension with absolute discriminant $D$ and degree $n = [E:F]$. Then the size $h_\ell(E)$ of the $\ell$-torsion in the class group of $E$ is
	\[h_\ell(E)= O_{\epsilon,n,\ell,F}(D^\epsilon)\mbox{ for all }\epsilon>0.\]
\end{conjecture}
We write $\Gal(E/F)=G \leq S_n$ if the Galois group of the Galois closure $\hat{E}/F$ of $E/F$ viewed as
permutation group on the set of embeddings of $E$ into $\hat{E}$ is permutation isomorphic to $G$.
In our second conjecture we consider the function
\[N(F,G;X):=\#\{E/F\mid \Gal(E/F)=G, \Disc(E/F)\leq X \}.\]  By a famous result of Hermite (e.g. see \cite[Theorem 2.24, page 68]{Nar04}) there are only finitely many number fields with the same given discriminant. 
\begin{definition}\label{def:ind}
	Let $G\leq S_n$ be a transitive group acting on $\Omega=\{1,\ldots,n\}$.
	\begin{enumerate}
		\item For $g\in G$ we define the index $\ind(g):= n- \mbox{ the number of orbits of $g$ on }\Omega.$
		\item For $n>1$ let $a(G):=\ind(G):=\min\{\ind(g): \id\ne g\in G\}.$
	\end{enumerate}
\end{definition}
Note that $a(G)$ here is the inverse of the $a(G)$ defined in \cite{Mal02}.
\begin{conjecture}[Malle's Conjecture (weak version of the upper bound), \cite{Mal02}]\label{conj:Mal1}
	Let $F$ be a number field and  $G\leq S_n$ be a transitive group. Then we have
	\[N(F,G;X)=O_{\epsilon,F}(X^{1/a(G)+\epsilon}) \mbox{ for all }\epsilon>0.\]
\end{conjecture}
In the same paper Malle conjectures a lower bound which is equivalent to \[\liminf_{X\rightarrow\infty} X^{-1/a(G)}N(F,G;X)>0.\]
We remark that Malle gives a refined version of the conjecture in \cite{Mal04}, which we do not need in our context here. There are also counter-examples known due to the first author \cite{Kl05} for this refined conjecture.

In our last conjecture we consider the number
\[a_D:=\#\{E/F\mid \Gal(E/F)=G, \Disc(E/F)=D \}\]
 of $G$-extensions of $F$ with discriminant $\Disc(E/F)=D$.
\begin{conjecture}[Discriminant Multiplicity Conjecture, \cite{Duk98,EV05}]\label{conj:disc-multi}
	Let $F$ be a number field and $G\leq S_n$ be a transitive group. Then for all $D\in\N$ we have
		\[a_D = O_{\epsilon,F,n}(D^{\epsilon})\mbox{ for all }\eps>0.\]
\end{conjecture}

The goal of this paper is twofold: on one hand, to show the relations among Conjectures \ref{conj:ell-torsion}, \ref{conj:Mal1}, and \ref{conj:disc-multi}; on the other hand, to show that these conjectures have affirmative answers when we restrict our discussion to certain general families of number fields. 

These conjectures share the common feature that they are all about giving upper bounds on the number of arithmetic objects, including number fields and class numbers. We will show that they are almost all equivalent to each other.
\begin{proposition}\label{prop1}
	Let $F$ be a number field. Then
	\begin{enumerate}
		\item Assume that Conjecture \ref{conj:disc-multi} is true for all $G\leq S_{n\ell}$, then Conjecture \ref{conj:ell-torsion} is true for $\ell$ and for all extensions $E/F$ of degree $n$.
		\item Assume that Conjecture \ref{conj:ell-torsion} is true for all solvable extensions $E/F$ and all prime numbers $\ell$. Then Conjecture \ref{conj:Mal1} is true for all solvable groups $G$.
	\end{enumerate}
\end{proposition} 
\begin{proof}
	Part (i) is shown in \cite[p. 164]{EV05} and \cite[Thm 1.7]{PTBW2}.  The second part is \cite[Cor. 1.6]{Alb}.
\end{proof}
  A little bit stronger Alberts proves in \cite[Corollary 1.4]{Alb} that Conjecture \ref{conj:Mal1} is true for solvable groups, if we assume that the torsion conjecture is true in average.   

Noticing that Conjecture \ref{conj:Mal1} is essentially an average statement of  Conjecture \ref{conj:disc-multi}, it is not surprising that in general Conjecture \ref{conj:disc-multi} implies Conjecture \ref{conj:Mal1}. We manage to prove that they are equivalent.
\begin{theorem}\label{thm:1}
	Let $F$ be a number field. Then 
\begin{enumerate}
\item 
Conjecture \ref{conj:Mal1} for all finite groups $G$ implies Conjecture \ref{conj:disc-multi} for all finite groups $G$.
\item Conjecture \ref{conj:disc-multi} for $G\leq S_n$ implies 
Conjecture \ref{conj:Mal1} for $G$.
\end{enumerate}	
\end{theorem}
The first part is shown in Theorem \ref{thm:C-B} and the second part in Theorem \ref{thm:malle-2-disc-multi}.

We remark that Proposition \ref{prop1} (i) and Theorem \ref{thm:1} (i) will be also true for the class of solvable extensions, see Corollary \ref{cor:disc}. Therefore we get: 
\begin{corollary}
	Conjectures \ref{conj:ell-torsion}, \ref{conj:Mal1}, and \ref{conj:disc-multi} are equivalent when we restrict to solvable extensions $E/F$.
\end{corollary}
We remark that we do not expect  that the similar statement for nilpotent extensions $E/F$ is true. The reason is that if you want to consider $\ell$--torsion  of the class group for $p$--extensions, then the resulting Galois groups are solvable, but in most cases not nilpotent.

We then focus on proving some special cases of these conjectures. Conjecture \ref{conj:ell-torsion} has a clear dichotomy depending on whether $\Cl_E[\ell]$ is randomly distributed. Classically, it is only known to be true when $F=\Q$ and $\ell = 2$ for $[E:F] = 2$ by Gauss using genus theory. Firstly, we give a compact proof on all cases for Conjecture \ref{conj:ell-torsion} where a similar argument with $\ell=2$ for quadratic extensions applies, i.e., when the distribution of $\Cl_E[\ell]$ is governed by genus theory. We say a field extension $E/F$ is an $\ell$-extension when  $\Gal(E/F)$ is an $\ell$-group. 
\begin{theorem}\label{thm:ell-torsion-ell-extension}
	Conjecture \ref{conj:ell-torsion} holds for the $\ell$-torsion of class groups of $\ell$-extensions.
	\end{theorem}
We will prove a more precise version in Theorem \ref{thm:main1}.
 Results on Conjecture \ref{conj:ell-torsion} for $\ell$, where $E/F$ is not an $\ell$--extension are much more difficult to prove. 
 Heuristically, Conjecture \ref{conj:ell-torsion} is shown to be a consequence of the moments version of the Cohen-Lenstra heuristics in \cite{PTBW2}. We mention some results in this direction: $\ell = 2$ \cite{BSTTTZ}, $\ell = 3$ with $d\le 4$ \cite{EV05}, and for arbitrary $\ell$ with $\Gal(E/F)= (\Z/p\Z)^r$ ($r>1$) \cite{JW20}. There are also recent works \cite{Chen,Ellen16,FreWid18,FreWid18x,ML17,ZTArtin,Widmer1} on proving a non-trivial bound on the $\ell$-torsion of class groups for almost all number fields in some certain family of number fields. 

In Section \ref{sec:multiplicity} we prove the following theorem by applying Theorem \ref{thm:ell-torsion-ell-extension}. 
\begin{theorem}\label{thm:disc-multi-nilpotent}
	Conjecture \ref{conj:disc-multi} holds for nilpotent number field extensions $E/F$.
\end{theorem}
Finally, by applying Theorems \ref{thm:disc-multi-nilpotent} and \ref{thm:1}, we recover the following theorem.
\begin{theorem}  \label{thm:weakmalle}
	Conjecture \ref{conj:Mal1} holds for nilpotent number field extensions $E/F$. 
\end{theorem}
We mention that Theorem \ref{thm:weakmalle} is also proved in \cite{KlMa04} for Galois nilpotent extensions and in \cite{Alb} for general nilpotent extensions. We recover this theorem as a direct consequence of Theorem \ref{thm:disc-multi-nilpotent}, and therefore give a short and simplified proof.

We finally remark that, as shown in the proof of Theorem \ref{main0}, the fundamental reason for these conjectures to hold in such a perfect shape in these cases is completely group theoretic, i.e., nilpotent groups have non-trivial center.

All results in this paper are effective.

\section{$\ell$-torsion conjecture}\label{sec:torsion}
In this section we prove Theorem \ref{thm:ell-torsion-ell-extension}. We give a more detailed version in   Theorem \ref{thm:main1}. 
We start with the following theorem, which is proved in \cite[Theorem 2.2]{Cor} for odd $\ell$ and generalized in \cite[Theorem 2]{Ros} to $\ell=2$.  In order to keep this
note self-contained we give a proof of this statement here. In the following we use the notion places for finite prime ideals and infinite places.
\begin{theorem}\label{main0}
  Let $E/F$ be a cyclic extension of number fields of degree $\ell$ ramified in $t$ places. Let $e:=\max(t,1)$. Then
     $$\rk_\ell(\Cl_E)\leq \ell(e-1 + \rk_\ell(\Cl_F)).$$  
\end{theorem}
\begin{proof}
	Firstly, the Galois group $\Gal(E/F) = \langle \sigma \rangle$
        acts on the $\F_{\ell}$-vector space $\Cl_E[\ell]$. Since
        $\sigma^{\ell} = \text{id}$, the minimal polynomial for
        $\sigma$ on $\Cl_E[\ell]$ divides $x^\ell-1 = (x-1)^{\ell}$.
        Therefore, the only eigenvalue is $1$ and each
        Jordan block has at most size $\ell$. It suffices to prove that
        the number of Jordan blocks is bounded by
        $e-1+\rk_{\ell}(\Cl_F)$.
	
	The number of Jordan blocks is equal to the dimension of the
        maximal quotient space on which $\sigma$ acts trivially.
Denote the corresponding class field by $M$. Notice that since $\sigma$ acts trivially, the field $M/F$ is Galois and abelian, therefore $\Gal(M/E)\cong C_\ell^s$ for some $s\geq 0$ and $\Gal(M/F)\cong C_\ell^{s+1}$ or
$\Gal(M/F)\cong C_{\ell^2}\times C_\ell^{s-1}$. We would like to prove that
$s\leq e-1+\rk_{\ell}(\Cl_F)$. We note that the second case can only happen
when $t=0$, i.e. $E/F$ is unramified. In this case we see $s=\rk_\ell(\Cl_F)$
and our claim is proved.

In the first case denote by $L/F$ the maximal unramified (including infinite places) subextension
of $M/F$. We know by construction that $\Gal(L/F)\cong
C_\ell^{\rk_\ell(\Cl_F)}$.  $M/L$ is abelian and it has no subextension which is everywhere unramified (including infinite places).
Therefore $\Gal(M/L)$ is generated by the inertia groups of the ramified
prime ideals including the infinite ones. Let $\idp$ be a prime ideal of $\OO_F$ which is ramified
in $M$. Since $M/E$ is unramified, we see that the inertia group has
size $\ell$ and is therefore cyclic. The same applies for the prime
ideal in $L$ lying above $\idp$. The inertia groups at infinite places are always cyclic and we see that each ramified prime in
$E/F$ can increase the rank of $\Gal(M/L)$ by at most 1. We therefore
get:
\[\rk_\ell(\Gal(M/F)) \leq \rk_\ell(\Cl_F) + e \mbox{ and }\rk_\ell(\Gal(M/E))=\rk_\ell(\Gal(M/F))-1.\qedhere\]
\end{proof}

\begin{remark} \label{rem:up}
	In case  $\rk_\ell (\Cl_F)=0$ \cite[Theorem 2.2]{Ros} and \cite[Theorem 2]{Cor} prove a slightly better upper bound for $\rk_\ell(\Cl_E)$, i.e. 
	$\rk_\ell(\Cl_E) \leq (\ell-1) \cdot (e-1).$ 
	It is known that for $F=\Q$ and $\ell=2$ this bound is sharp by genus theory.
	Furthermore, \cite[Theorem 2.7]{Cor} also gives a better bound for the cyclic of order $\ell^r$-case compared to the inductive approach we present in Theorem \ref{main1}.
\end{remark}
We do not prove this remark since we are only interested in the asymptotic behavior and therefore the change of constants does not matter. 	
In order to prove Theorem \ref{main1} for non normal extensions we need the following lemma.	
\begin{lemma} \label{lem:tower}
	 Let $n=\ell^r$ and $G\leq S_n$ be an $\ell$-group and $E/F$ be an extension of number fields with $\Gal(E/F)\cong G$. Then there exists a tower of fields
	 \begin{equation}\label{eq:1}
	 F=F_0 \leq F_1\leq \ldots \leq F_{r-1} \leq F_r=E
	 \end{equation}
	 such that $\Gal(F_{i+1}/F_i)=C_\ell$ for all $0\leq i \leq r-1$.
\end{lemma}
\begin{proof}
Let $\tilde{E}$ be the normal closure of $E$ over $F$. Denote $H$ to be the subgroup of $G$ fixing $E$ and choose a maximal subgroup $G_1\leq G$ that contains $H$. Note that all maximal subgroups of an $\ell$-group have index $\ell$ and are normal. Define $F_1$ to be the subfield of $\tilde{E}$ fixed by $G_1$. Inductively, we can find a sequence of subgroups $G\cong G_0\supset G_1\supset \cdots\supset G_r = H$ with $[G_i:G_{i+1}]= \ell$ for every $0\le i\le r-1$ and define $F_i$ to be the subfield fixed by $G_i$. 
\end{proof}
	
 Now we prove our main result of this section. We remark that \cite[page 424]{Cor} describes just before Theorem 3 how to get this result for normal $\ell$-extensions.

\begin{theorem}\label{main1}
	Let $n=\ell^r$, $G\leq S_n$ be a transitive $\ell$-group, and $E/F$ be an ex\-ten\-sion of number fields with $\Gal(E/F)\cong G$, 
  and with tower as defined in \eqref{eq:1}.
  Let $t_i$ be the number of ramified places in $F_{i+1}/F_i$ and 
  $e_i:=\max(t_i,1)$. Then we get:
    \begin{equation}\label{eq:2}
    \rk_\ell(\Cl_E)\leq \sum_{i=0}^{r-1} \ell^{r-i}(e_i-1) + n\rk_\ell(\Cl_F).
    \end{equation}
\end{theorem}
\begin{proof}
By Lemma \ref{lem:tower} we find a tower of cyclic extensions of order $\ell$.
The assertion now follows by applying Theorem \ref{main0} for each step.	
\end{proof}	
The above version is still a little bit complicated since we need to
know the number of ramified places in each step. It would be
much nicer to have a bound which is only depending on the number of
ramified places of $F$. Let $\idp$ be a prime
ideal of $\OO_F$  which ramifies for the first time
in $F_{i+1}$. We have the extreme case if $\idp$ splits completely in $F_i$
which means that there are $\ell^i$ places over $\idp$ lying in $F_i$.

Therefore, if $t$ is the number of prime ideals in $\OO_F$ which are
ramified in $E/F$, then we get that $t_i\leq e_i \leq \max(\ell^i t,1)$ in
\eqref{eq:2}. Therefore we get:
\begin{lemma}\label{lem}
  Let $G\leq S_n$ be a transitive $\ell$-group and $E/F$ be an extension with $\Gal(E/F)\cong G$ of degree $n=\ell^r$ which is
   ramified in $t$ places. Then
   \begin{equation}\label{eq:3}
    \rk_\ell(\Cl_E)\leq  r n t + n \rk_\ell(\Cl_F).
   \end{equation}
\end{lemma}
\begin{proof}
  Note that $e_i-1 \leq \ell^i t$ and using this in \eqref{eq:2} we get
  $\rk_\ell(\Cl_E)\leq  r n t + n\rk_\ell(\Cl_F)$
  and the assertion follows easily.
\end{proof}

In the next step we would like to know an upper bound for the number
of different prime ideals dividing the discriminant of $E/F$. We use
the following standard result, e.g. see \cite[Section 5.3, p. 83]{Ten}.
\begin{proposition}\label{omega}
  For an integer $n$ we denote by $\omega(n)$ the number of distinct
  prime factors. Then there exists an explicit constant $C>0$ such that for every $n> 2$, 
  $$\omega(n)\leq C\frac{\log n}{\log\log n}.$$
  Let $F$ be a number field of degree $d$.
  Then for an integral ideal $\mathfrak{n}\normal O_F$ with absolute norm $n = |\mathfrak{n}|>2$, the number $\omega(\mathfrak{n})$ of distinct prime ideal factors is bounded by
  $$\omega(\mathfrak{n}) \leq d \cdot \omega(n) \le Cd\frac{\log n}{\log\log n}.$$
\end{proposition}
We remark that the average order of $\omega(n)$ is $\log\log n$. Note that we can choose $C=1.3841$, see
\cite[Thm. 11]{Rob}.

Using that the number of ramified prime ideals in a relative extension $E/F$ is $\omega(\disc(E/F))$, we prove our main result by applying Proposition \ref{omega} and Lemma \ref{lem}. 
\begin{theorem}\label{thm:main1}
  Let $E/F$ be an $\ell$-group extension of degree $n=\ell^r$ and absolute discriminant
  $D:=\Disc(E/F)$, and define $d:=[F:\Q]$. Then we get:
  $$\rk_\ell(\Cl_E) \leq n \rk_\ell(\Cl_F) + nr\cdot C d \frac{\log D}{\log\log D} \text{ for }D>2,$$
 equivalently, we get for the size $h_\ell(E)$ of the $\ell$-torsion part $\Cl_E[\ell]$:
    $$h_{\ell}(E) \leq  h_{\ell}(F)^n \cdot D^{\frac{Cndr\log \ell}{\log\log D}} = O_{\epsilon,F, n}(D^\eps)
  \mbox{ for all }\eps>0.
  $$
  For $D=1$ we get $\rk_\ell(\Cl_E) \leq n \rk_\ell(\Cl_F)$ and therefore $h_{\ell}(E) \leq  h_{\ell}(F)^n$.
 \end{theorem}
Note that the case $D=2$ is not possible by Hilbert's ramification theory. The reader can also find an independent proof of Theorem \ref{thm:main1} by G. Gras \cite[pp. 2 and 9]{Gras}, which he gave after seeing our preprint. We remark that the Galois group of all fields in his family ${\mathcal F}_K^{p^e}$ are also $p$--groups.

 \begin{remark}
 	Note that we easily get the following estimate for the $\ell^s$-torsion $$h_{\ell^s}(E) \le h_{\ell}(E)^s \le h_{\ell}(F)^{ns} \cdot D^{\frac{Cndrs\log \ell}{\log\log D}} =  O_{\epsilon,F,n,s}(D^{\epsilon}).$$
 \end{remark}

Theorem \ref{main1} or Lemma \ref{lem} are not expected to be sharp, but will be sufficient for our purpose, since we only aim at proving Conjecture \ref{conj:ell-torsion}. However, it is also an independent interesting question to study the upper bound at a finer scale. We mention results along this direction \cite{KoPa19,KoPa20} for $F=\Q$, where a sharp upper bound is obtained for $\ell=2$ and certain special family of multi-quadratic number fields. 




\section{Discriminant Multiplicity Conjecture for Nilpotent Extensions}\label{sec:multiplicity}
The goal of this section is to prove Theorem \ref{thm:disc-multi-nilpotent}  which answers Conjecture \ref{conj:disc-multi} positively in the nilpotent case. As usual we can reduce the nilpotent case to the $\ell$--group case, but we have to be a little bit careful (see Lemma \ref{lem:p-grp-2-nilpotent}) that this reduction is compatible with permutation groups. For the $\ell$--group case we use the upper bound proved in Theorems \ref{main1} and \ref{thm:main1}. 
In a first step we prove this theorem for $\ell$-groups and then use this for proving the case of arbitrary nilpotent groups. For the latter step
 we need a group theoretic lemma. This states that any transitive nilpotent permutation group $G\leq S_n$ is isomorphic to a natural direct product of its $\ell$-Sylow subgroups. It is a standard fact that all nilpotent groups are isomorphic to the direct product of their $\ell$-Sylow subgroups, however we emphasize that we need to prove the isomorphism in the category of \textit{permutation groups}. Equivalently, this means that all nilpotent $G$-extensions (not necessarily Galois) can be realized as a compositum of $\ell$-extensions. 

\begin{lemma}\label{lem:p-grp-2-nilpotent}
	A transitive nilpotent permutation group $G\leq S_n$ is permutation isomorphic to the natural direct product of transitive permutation groups $G_\ell\leq S_{n_\ell}$,
	$$G \simeq \prod_\ell G_\ell \mbox{ with } n=\prod_\ell n_\ell,$$
	where the $G_\ell$ are isomorphic to the $\ell$-Sylow subgroups of $G$ and $n_\ell$ is the maximal $\ell$-power dividing $n$.
\end{lemma}
\begin{proof}
	Firstly, it is a standard result that a nilpotent group $G$ is equal to the direct product $\prod_{\ell\mid n} \Syl_\ell(G)$ where $\Syl_\ell(G)$ is the $\ell$-Sylow subgroup of $G$. Let $H\leq G$ be a stabilizer of a point. This means that $G$ can be realized by the action of $G$ on the left cosets of $G/H$.
	Now $H$ is nilpotent as well and therefore it is a direct product of its Sylow subgroups. We get:
	\[H=\prod_{\ell\mid n} \Syl\nolimits_\ell(H) \text{ with }\Syl\nolimits_{\ell}(H) \leq \Syl\nolimits_{\ell}(G).  \]
	Now we can define the permutation groups $G_\ell$ by the action of $\Syl_\ell(G)$ on the left cosets of $\Syl_\ell(G)/\Syl_\ell(H)$ for each prime $\ell$ dividing $n$.
	Since $H=\prod_{\ell\mid n} \Syl\nolimits_\ell(H)$ we see that $G$ is permutation isomorphic to the natural direct product $\prod_{\ell\mid n} G_\ell \leq S_n$ of permutation groups. 
%
%
%
\end{proof}

\begin{lemma}\label{specialcase}
	Let $F$ be a number field of degree $d$, $\ell$ be a prime number, and $\disc$ be an ideal of $\OO_F$. Then the number of $C_\ell$-extensions $E/F$ with $\disc(E/F)=\disc$
	is bounded above by
	\[ O_{d, \ell}(h_{\ell}(F)\cdot \ell^{\omega(\disc)})=O_{\epsilon, d, \ell}( h_{\ell}(F) \cdot D^{\epsilon}) = O_{\epsilon, F,\ell}( D^{\epsilon}) \mbox{ for all }\eps>0.\]   
\end{lemma}
\begin{proof}
	Let $E/F$ be such an extension. Then by class field theory the finite part $\idf_0$ of the conductor has the property that $\disc=\idf_0^{\ell-1}$. Denote
	by $\idf_\infty$ the set of real places, if $\ell=2$ and let $\idf_\infty=\emptyset$ otherwise.
	Then $E$ is contained in the ray class field of $\idf=\idf_0\idf_\infty$ and we need an upper bound on the size of the $\ell$-torsion of this ray class group $\Cl_\idf$.
	By class field theory we have the following short exact sequence: 
$$ \prod_{\idp\mid \idf_0} (\OO_F/\idp^{e_\idp})^* \times \prod_{\idp\mid \idf_\infty} C_2 \to \Cl_\idf \to \Cl_F \to 0\;\; \mbox{ with }\idf_0=\prod_{\idp\mid \idf_0} \idp^{e_\idp}.$$
Therefore the $\ell$-rank of $\Cl_\idf$ is bounded above by $\sum_{\idp\mid\idf_0} \rk_\ell((\OO_F/\idp^{e_\idp})^*) + d + \rk_\ell(\Cl_F)$. Note that
$\rk_\ell((\OO_F/\idp^{e_\idp})^*) \leq 1$ for $\ell\not\in\idp$ and $\rk_\ell((\OO_F/\idp^{e_\idp})^*) \leq [F_\idp:\Q_\ell]+1$, if $\ell\in \idp$.
Considering the wildly ramified primes, the extreme case happens when all wildly ramified primes are dividing $\idf_0$ and $\ell$ splits in $F$. In this
situation the wildly ramified primes might increase the $\ell$-rank by $2d$. The infinite places might increase the 2-rank by $d$ and therefore we get the following upper bound:
\[\rk_\ell(\Cl_\idf)\leq \omega(\idf_0)+3d+\rk_\ell(\Cl_F).\]

Therefore the number of $C_{\ell}$-extensions $E/F$ with $\disc(E/F) = \disc$ is bounded by $|\Cl_\idf[\ell]|=O_{d, \ell}(h_{\ell}(F)\cdot \ell^{\omega(\disc)})=O_{\epsilon, d, \ell}( h_{\ell}(F) \cdot D^{\epsilon}) = O_{\epsilon, F,\ell}( D^{\epsilon})$ for all $\eps>0$ by Proposition \ref{omega}. 
	\end{proof}

\begin{proof}[Proof of Theorem \ref{thm:disc-multi-nilpotent}]
	Denote by $b_D$ the number of ideals $\disc$ of $\OO_F$ such that $|\disc|=D$. We claim that $b_D= O_{d}(C^{\omega(D)})$ for some $C$ depending on the degree $d$. This is bounded by $O_{\epsilon, d}(D^{\epsilon})$. Therefore it suffices to prove that the number of $G$-extensions $E/F$ with $\disc(E/F) = \disc$ is bounded by $O_{\epsilon, F, n}(D^{\epsilon})$.
	
	In order to prove the claim note that $b_D$ is multiplicative and therefore it suffices to prove it for prime powers $D=p^s$. The worst case happens
	when $p$ is split in $F$. 
	Then the number of ideals is equal to $\binom{d+s-1}{d-1}\leq (s+1)^{d-1}$ and $s=O_d(\log D)$ which gives $\binom{d+s-1}{d-1} = O_d((\log D)^{d-1})=O_{\epsilon, d}(D^{\epsilon})$ for all $\eps>0$.

	Let us assume that $G\leq S_n$ is a transitive $\ell$-group of degree $n = \ell^r$. We proceed by induction on $r$. When $r = 1$, then $G=C_\ell$ and we apply Lemma \ref{specialcase}.
	
	Suppose the statement holds for $\ell$-extensions of degree $n = \ell^r$. Given an arbitrary $\ell$-extension $E/F$ of degree $\ell^{r+1}$, there is a chain of subfields $E = E_{r+1} \geq E_{r} \geq \cdots \geq E_{0} = F$ using Lemma \ref{lem:tower}. 
	Denote $\disc(E_r/F)= \fm$, then we have $\fm^\ell \cdot \fm' = \disc$ by the discriminant formula for towers. 
	
	By induction for $n=\ell^r$, the number of extensions $E_{r}/F$ with $\disc(E_r/F) = \fm \mid \disc$ is bounded by $O_{\epsilon, F, n}( |\fm|^{\epsilon} )$. Using Lemma \ref{specialcase} the number of $E_{r+1}/E_r$ with relative discriminant $\Nm_{E_r/F}(\disc(E_{r+1}/E_r)) = \fm'$ is bounded by $O_{\epsilon, d, \ell} (h_\ell(E_r) \cdot |\fm'|^{\epsilon} )$. Since $h_\ell(E_r)=O_{\eps,F,n}(|\fm|^\eps)$ by Theorem  \ref{thm:main1}, we get the bound $O_{\epsilon, F, n}( |\fm'|^{\epsilon} |\fm|^{\epsilon})$ for the number of $E_{r+1}/E_r$ with relative discriminant $\Nm_{E_r/F}(\disc(E_{r+1}/E_r)) = \fm'$.
	
	  Therefore for each $\fm^\ell\mid \disc$, the number of extensions $E_{r+1}/F$ with $\disc(E_{r+1}/F)=\disc$ and $\disc(E_r/F)= \fm$ is bounded by $O_{\epsilon, d, n} (D^{\epsilon})$. The number of divisors $\fm\mid \disc$ is bounded by $O_{\epsilon}(D^{\epsilon})$. So the number of $E_{r+1}/F$ with $\disc(E_{r+1}/F) =\disc $ in total is bounded by $O_{\epsilon, F, n}( D^{\epsilon})$. This finishes the proof of the discriminant multiplicity conjecture for general $\ell$-extensions. 
	
	Secondly, we deduce the discriminant multiplicity conjecture for nilpotent extensions from the one for $\ell$-extensions. 
	Given a transitive nilpotent permutation group $G\leq S_n$, by Lemma \ref{lem:p-grp-2-nilpotent}, we have 
	\[G \simeq \prod_{i=1}^s G_{\ell_i} \leq \prod_{i=1}^s S_{\ell_i^{s_i}} \leq S_n \mbox { for }n=\prod_{i=1}^s \ell_i^{s_i}.
	\]  Therefore each $G$-extension $E/F$ is the compositum of $\ell_i$-extensions $E_{\ell_i}/F$ with $\Gal(E_{\ell_i}/F) = G_\ell\leq S_{\ell_i^{s_i}}$. Therefore the number of $G$-extensions $E/F$ with $\disc(E/F)=\disc$ is bounded by the number of tuples $(E_{\ell_1},\ldots,E_{l_s})$ of $\ell_i$-extensions with $\disc(E_{\ell_i}/F) \mid \disc$ and $\Gal(E_{\ell_i}/F) = G_{\ell_i}$ for each $i$. Combining the result on $\ell$-extensions and the fact that the number of divisors of $\disc$
	 is bounded by $O_{\epsilon, n}(D^{\epsilon})$, we deduce that the number of $E_\ell/F$ for each prime $\ell\mid n$ is bounded by $O_{\epsilon, F, n_{\ell}}( D^{\epsilon})$. Taking the product over all $\ell \mid n$, we get that the number of such tuples is bounded by 
	 $O_{\epsilon, F, n}( D^{\omega(n)\epsilon}) = O_{\epsilon, F, n}(D^{\epsilon})$.
\end{proof}

\section{Malle's Conjecture for Nilpotent Extensions}\label{sec:Malle}
The goal of this section is to prove Theorems \ref{thm:1} and \ref{thm:weakmalle}. The proof of Theorem \ref{thm:1} is split into proving Theorems \ref{thm:C-B} and \ref{thm:malle-2-disc-multi}. The proof of Theorem \ref{thm:weakmalle} is the content of Corollary \ref{cor:weakmalle}.
\begin{theorem}\label{thm:C-B}
	Let $G\leq S_n$ be a transitive permutation group  and $F$ be a number field. Assume that Conjecture \ref{conj:disc-multi} is true for $F$ and $G$. Then Conjecture \ref{conj:Mal1} is also true for $F$ and $G$, i.e. the number of $G$-extensions $E/F$ with $\Disc(E/F)\le X$ is bounded above by:
	$$N(F,G; X) = O_{F, \epsilon,n}( X^{1/a(G)+\epsilon}) \mbox{ for all }\eps>0. $$
\end{theorem}
\begin{proof}
	Let $\mathcal{A} : = \{ D\in \mathbb{Z} : p \mid D \implies p^{a(G)} \mid D  \}$ and $\mathcal{A}(X):= \sharp \{ D\in \mathcal{A} : D<X \}$ be the counting function associated to $\mathcal{A}$. Let $a_D$ be the number of fields $E/F$ with Galois group $G$ and $D=\Disc(E/F)$. Note that $a_D=0$, if $D\not\in \mathcal{A}$, see \cite[Section $7$]{Mal02}. Then for any $\epsilon>0$, there exists a constant $C_{F, \epsilon,n}$ 
	by Conjecture \ref{conj:disc-multi} such that
	$$N(F, G;X) = \sum_{D\in \mathcal{A},D<X}  a_D \le \sum_{D\in \mathcal{A},D<X} C_{F, \epsilon,n}\cdot D^{\epsilon} \le C_{F, \epsilon,n}\cdot \mathcal{A}(X)\cdot X^{\epsilon}.$$
    We will show that
	\begin{equation}\label{eqn:aG-integer}
	\mathcal{A}(X) \sim C'\cdot X^{1/a(G)},
	\end{equation}
	for some $C'>0$, therefore $\mathcal{A}(X) = O(X^{1/a(G)})$. The generating series $f(s)$ of $\mathcal{A}$ is 
	$$f(s) = \prod_p (1+ \sum_{k\ge a(G)} p^{-ks}) = \zeta(a(G)s) \cdot g(s),$$
	where
	$$g(s) = \prod_p (1+ \sum_{a(G)+1 \le k\le 2a(G)-1} p^{-ks}),$$
   is a holomorphic function when $\Re(s)> 1/(a(G)+1)$. The function $f(s)$ thus has an analytic continuation to $\Re(s)\ge 1/a(G)$ except for a simple pole at $s = 1/a(G)$. We get (\ref{eqn:aG-integer}) from a Tauberian theorem, e.g. see \cite[p. 121]{Nar83}. Then the result follows since $C_{F, \epsilon,n}\cdot \mathcal{A}(X) \cdot X^{\epsilon} = O_{F,\epsilon,n}(X^{1/a(G)+\epsilon})$. 		
%
\end{proof}
	Note in the above proof that $\mathcal{A}(X)=O_{\epsilon}(X^{1/a(G)+\epsilon})$ for all $\epsilon>0$ can be easily derived without using a Tauberian theorem. This is certainly sufficient to finish the proof.
Using Theorems \ref{thm:disc-multi-nilpotent} and \ref{thm:C-B} we immediately get the following corollary which is the content of Theorem \ref{thm:weakmalle}.
\begin{corollary}\label{cor:weakmalle}
Let $G$ be a transitive nilpotent group. Then Conjecture \ref{conj:Mal1} is true for all $\epsilon>0$ and all number fields $F$.
\end{corollary}

Now we prove some version of the inverse implication of Theorem \ref{thm:C-B}.
We cannot expect that Conjecture \ref{conj:Mal1} for a single $G$ implies Conjecture \ref{conj:disc-multi} for the same $G$, because it is easy to write down series of numbers $a_m$ which have a given asymptotic behavior, but a single $a_m$ is not bounded by $O(m^\epsilon)$ for all $\epsilon>0$. E.g. we can take the series 
\[a_m = \begin{cases} \sqrt{m} & \text{if $m$ is a square,}\\ 0 &\text{otherwise}.\end{cases} \text{ and we see }\sum\limits_{m\leq x} a_m \sim x/2,  \]
but $a_m=\sqrt{m}$ for infinitely many $m$. In order to prove Conjecture \ref{conj:disc-multi} we need some control over the higher moments. We can get this control when we assume Conjecture \ref{conj:Mal1} for so-called allowable permutation groups (see Definition \ref{def:allow}), which are suitable subgroups of the direct product $G^k$. Using this assumption we are able to prove in Theorem \ref{thm:malle-2-disc-multi} that Conjecture \ref{conj:disc-multi} holds. 

 Let us consider $G$-extensions $E_i/F$ for $1\leq i\leq k$ for a transitive $G\leq S_n$. 
 We are interested in the Galois group of the tensor product $E= \otimes_{i=1}^k E_i$ which is a subgroup of $G^k$. Note that $K=K_1\cdots K_m$ is a product of fields and for $m=1$ it can be interpreted as the compositum of the $E_i$. We define $\rho: G_{F} \to G^k \leq S_{n^k}$, where $G_F$ denotes the absolute Galois group of $F$. The image $\rho(G_{F})$ might not be transitive of degree $n^k$, equivalently, this means that $K$ is not a field.
 Note that the number of orbits of $\rho(G_{F})$ is equal to $m$ and if ordered in the right way we get that $|O_i|=[K_i:F]$.
   Choosing one orbit $O:=O_i$ of size $nd$ and the corresponding field $E:=K_i$ we get a permutation representation for $\Gal(E/F)\leq S_{nd}$. 

\begin{definition} \label{def:allow}
	Given a transitive $G\leq S_n$ and an integer $k$. We say $U\leq G^k$ is an \emph{allowable subgroup} of the natural direct product $G^k$ if for every $1\le i\le k$, we have $\pi_i(U)= G$ where $\pi_i: G^k \to G$ is the natural projection to the $i$-th component. If $O\subseteq \{ 1, \cdots, n^k \}$ is an orbit of $U$ of size $nd$, we will say that the transitive  action of $U$ on $O$ is an \emph{allowable permutation subgroup} $H\leq S_O$ of $G^k$.
\end{definition}


\begin{theorem}\label{thm:malle-2-disc-multi}
	Let $F$ be a number field and $G\leq S_n$ be a transitive group. Then the correctness of Conjecture \ref{conj:Mal1} for $F$ and for all allowable permutation subgroups $H$ of $G^k  \leq S_{n^k}$ for all $k>0$ implies the correctness of Conjecture \ref{conj:disc-multi} for $F$ and $G$.
	 \end{theorem}

 A similar idea was used in \cite[Prop. 4.8]{EV05} for $F=\Q$.  Instead of counting by discriminants $D$ they assume a Malle conjecture for counting number fields by the radical of the discriminant.
 We need two lemmata before we can prove this theorem.
\begin{lemma}\label{lem:k-moments-uniformity}
	Given a sequence $\{ a_m\}$ of non-negative real numbers, then the following statements are equivalent:
	\begin{enumerate}
		\item $\forall \epsilon> 0: a_m =O_\epsilon(m^{\epsilon})$. 
		\item There exists an $A>0$ such that for all $k>0: \sum\limits_{ X< m\le 2X} a_m^k = O_k(X^{A}).$
		\item
		There exists an $A>0$ such that for all $k>0:
		\sum\limits_{ m\le X} a_m^k = O_k(X^{A}).$
	\end{enumerate}
\end{lemma}
\begin{proof}
	It is easy to see that $(ii)$ and $(iii)$ are equivalent. To go from $(ii)$ to $(iii)$ it suffices to add up over dyadic ranges. The other direction is immediate.

To go from $(i)$ to $(ii)$ is immediate for any $A>1$.   To go from $(ii)$ to $(i)$, for any fixed $\epsilon>0$, we apply the following Chebychev type inequality which is easy to see here. We get for all $k>0$: 
$$  X^{k\epsilon} \cdot \sum_{X<m\le 2X, a_m> X^{\epsilon}} 1\le  \sum_{X< m\le 2X, a_m> X^{\epsilon}} a_m^k \stackrel{(ii)}{=} O_k(X^A).$$
 Choosing $k> A/\epsilon$ we get that \[\sum_{X<m\le 2X, a_m> X^{\epsilon}} 1= O_k(X^{A-k\epsilon}),\] which must be $0$ when $X$ is large enough. Therefore for any $\epsilon>0$, there exists a $X>0$ such that when $m> X$, we have $a_m \le m^{\epsilon}$. 
\end{proof}

We remind the reader that $\ind(G)$ is defined in Definition \ref{def:ind}. We denote the degree of the permutation group $G$ by $\deg(G)$.
\begin{lemma}\label{lem:deg-over-ind-uniformity}
	Let $G\leq S_n$ be a transitive permutation group, $k$ be an integer and $H$ be an allowable permutation subgroup of $G^k$. Then we get the following inequality:
		$$\frac{ \deg}{\ind}(H) \le \frac{\deg}{\ind} (G).$$
\end{lemma}
\begin{proof}
Let $H$ be an allowable permutation subgroup of $G^k$ acting on $nd$ points. If $nd<n^k$ it might be that we only see $1\leq \ell \leq k$ projections to $G$.  Suppose that $g\in H$ is a non-trivial element. This implies that
$g$ is a non-trivial element in at least one projection $\kappa: H \rightarrow G$ and
we define $\bar g:=\kappa(g)$. By the definition of $\ind(G)$ we have that $\ind(\bar g)\geq \ind(G)$.  Let us look at all possible preimages of $\bar{g}$ under $\kappa$. The index of those elements will be minimal if each cycle of length $r$ of $\bar{g}$ will decompose into $d$ different cycles of length $r$. Therefore we get for a preimage $\tilde{g}$ of ${g}$ under $\kappa$:
\[ \#\{\text{cycles of }\tilde{g}\} \leq d \# \{\text{cycles of }\bar{g}\}.\]
Therefore we get: \[\ind(\tilde{g}) = nd - \#\{\text{cycles of }\tilde{g}\} \geq nd - d\#\{\text{cycles of }\bar{g}\}\] 
\[= d(n-  \#\{\text{cycles of }\bar{g}\}) = d\cdot\ind(\bar{g}).  \]
This line is equivalent to
\[\frac{1}{\ind(\tilde{g})} \leq \frac{1}{d\cdot\ind(\bar{g})} \Leftrightarrow
\frac{nd}{\ind(\tilde{g})} \leq \frac{n}{\ind(\bar{g})} = \frac{\deg(G)}{\ind(\bar{g})}\leq \frac{\deg(G)}{\ind(G)}.
  \]
  Note that $\deg(H)=nd$ and $\ind(\tilde{g})\leq \ind(H)$ in order to finish the proof.
\end{proof}
Note that in the special case $H=G^k$ the above proof shows 
\[\frac{ \deg}{\ind}(G^k) = \frac{\deg}{\ind} (G)\]
by using $\ind(g)=\ind(G)$ implies that $\ind((g,1,\ldots,1))=\ind(G^k)$.
\begin{proof}[Proof of Theorem \ref{thm:malle-2-disc-multi}]
	We will prove that there exist constants $C_1>0$ and $C_2>0$ such that for all $k>0$
	\begin{equation}\label{count}
	\sum_{D \le X} a_D^k \ll_{k} \sum_{H} N(F,H; C_2X^{C_1 a(H)})  \stackrel{\text{Conj. B}}{=}O_{F,\epsilon,k}(X^{C_1+\epsilon}),	
	\end{equation}
		where the (finite) summation goes over all allowable permutation subgroups $H$ of $G^k$.
	Then the result follows from condition (iii) in Lemma \ref{lem:k-moments-uniformity} by letting $a_D$ be the number of $G$-extensions $E/F$ with $\Disc(E/F) = D$.
	
	For proving \eqref{count}, it suffices to find constants $C_1>0$ and $C_2>0$ for a given allowable permutation subgroup $H\leq S_{nd}$ we consider composita $E:=\prod_{i=1}^k E_i$  such that $\Disc(E/F) \le C_2X^{C_1a(H)}$ when $\Disc(E_i/F) = D \le X$ for every $i$ and $\Gal(E/F) = H$. Now we study the discriminant $\Disc(E/F)$.  Recall by Hilbert's ramification theory that the valuation
	$v_\idp(\disc(E/F)) < [E:F]$ for an at most
	tamely ramified prime ideal $\idp\subseteq O_F$. The same estimate is true for all other prime ideals of $O_F$ lying over the same prime number $p=\idp \cap \Z$. Therefore we get $v_p(|\disc(E/F)|)\leq [E:F] [F:\Q]$.

	 This implies that $\Disc(E/F) \le C_2D^{[E:\Q]}$, where $C_2$ is a bounded factor coming from wildly ramified primes. Assuming $D\leq X$ and $\deg(H)=[E:F]$ we get: \[\Disc(E/F)
	   \le C_2X^{[E:\Q]} \le C_2 X^{[F:\Q]\ind(H) \cdot (\deg/\ind)(H)} \le C_2 X^{[F:\Q]a(H)\cdot(\deg/\ind)(G) }\] by Lemma \ref{lem:deg-over-ind-uniformity}. Now we see that we can take $C_1 = [F:\Q](\deg/\ind)(G)$.
\end{proof}
\begin{corollary}\label{cor:disc}
	Let $F$ be a number field.
	Assume that Conjecture \ref{conj:Mal1} is true for all solvable groups $G$. Then Conjecture \ref{conj:disc-multi} is true for all solvable groups.
 \end{corollary}
\begin{proof}
	Noting that all allowable (permutation) subgroups of $G^k$ are solvable when $G$ is solvable, the results follows directly from Theorem \ref{thm:malle-2-disc-multi}. 
	\end{proof}
In a similar way Proposition \ref{prop1} (i) can be proved for the class of solvable extensions. Here we use that subgroups of $C_\ell \wr H$ are solvable when $H$ is solvable.

\section*{Acknowledgment}
Wang is partially supported by Foerster-Bernstein Fellowship at Duke University, and would like to thank Melanie Matchett Wood for helpful conversations. The authors would like to thank for the hospitality of the Mathemati\-sches Forschungsinstitut in Oberwolfach and the organizers of the workshop Explicit Method in Number Theory
2018 where this collaboration begins. The authors would like to thank Manjul Bhargava for many helpful conversations during the time at Oberwolfach. The authors would like to thank Brandon Alberts and Gunter Malle for suggestions on an earlier draft. This project is accomplished during the Research in Pairs (RIP) program at Mathemati\-sches Forschungsinstitut in Oberwolfach in 2019, supported by the Volkswagen-Stiftung. 
\bibliographystyle{plain} 
\bibliography{lrank}
\end{document}